\documentclass[10pt]{article}

\usepackage{amsmath}
\usepackage{amsfonts}

\usepackage{amsthm}

\usepackage{hyperref}

\usepackage{cleveref}

\usepackage{mathtools}

\usepackage{microtype}

\newtheorem{theorem}{Theorem}[section]
\newtheorem{proposition}[theorem]{Proposition}
\newtheorem{corollary}[theorem]{Corollary}
\newtheorem{lemma}[theorem]{Lemma}
\newtheorem{definition}[theorem]{Definition}

\newtheoremstyle{remarkstyle}
{}
{}
{}
{}
{\bfseries}
{.} 
{7pt}
{}

\theoremstyle{remarkstyle}
\newtheorem{remark}[theorem]{Remark}
\newtheorem*{acknowledgements}{Acknowledgements}

\newcommand{\N}{\mathbb{N}}

\newcommand{\R}{\mathbb{R}}

\DeclareMathOperator{\tr}{tr}
\DeclareMathOperator{\id}{id}
\DeclareMathOperator{\vol}{vol}

\newcommand{\der}[2]{\frac{d #1}{d #2}}
\newcommand{\parder}[2]{\frac{\partial #1}{\partial #2}}

\DeclarePairedDelimiter{\norm}{\lVert}{\rVert}
\DeclarePairedDelimiter{\abs}{\lvert}{\rvert}
\DeclarePairedDelimiterX{\inner}[2]{\langle}{\rangle}{#1, #2}

\newcommand{\presuperscript}[2]{\prescript{#1}{}{#2}}

\title{Exponential convergence rate of the harmonic heat flow}
\author{Ivo Slegers}

\begin{document}
\maketitle
	
\begin{abstract}
We consider the harmonic heat flow for maps from a compact Riemannian manifold into a Riemannian manifold that is complete and of non-positive curvature. We prove that if the harmonic heat flow converges to a limiting harmonic map that is a non-degenerate critical point of the energy functional, then the rate of convergence is exponential (in the $L^2$ norm).
\end{abstract}

\section{Introduction} 
The harmonic heat flow was introduced by Eells and Sampson in \cite{EellsSampson}. They used it to prove one of the first general existence results for harmonic maps between Riemannian manifolds. Since then the harmonic heat flow has been an important tool in many existence results for harmonic maps. It has also been studied much as a subject of investigation in its own right.

Suppose $(M,g)$ and $(N,h)$ are Riemannian manifolds and $f \colon M \to N$ a smooth map. The harmonic heat flow is an evolution equation on one-parameter families of smooth maps $(f_t \colon M \to N)_{t\in [0,\infty)}$ that continuously deforms $f$ into a harmonic map. The parameter $t$ is often thought of as a time parameter. The harmonic heat flow equation is
\begin{align}\label{eq:heatflow}
\begin{split}
\der{f_t}{t} &= \tau(f_t)\\
f_0 &= f.
\end{split}
\end{align}
Here $\tau(f_t)$ is the tension field of $f_t$ (see \Cref{sec:heatflowpreliminaries}). Eells and Sampson prove in \cite{EellsSampson} (with contributions of Hartman in \cite{Hartman}) that if $M$ is compact and $N$ is complete and has non-positive curvature, then a solution of \Cref{eq:heatflow} exists for all $t\geq 0$. Moreover, if the images of the maps $f_t$ stay within a compact subset of $N$, then the harmonic heat flow converges, for $t\to\infty$, to a harmonic map $f_\infty \colon M\to N$ that is homotopic to $f$.

In this note we prove that when the limiting map satisfies a certain non-degeneracy condition (which will elaborated on in \Cref{sec:heatflowpreliminaries}), then the rate of convergence of the harmonic heat flow is exponential.

\begin{theorem}\label{thm:exponentialconvergence}
Let $(M,g)$ and $(N,h)$ be Riemannian manifolds with $M$ compact and with $N$ complete and of non-positive curvature. Let $(f_t)_{t\in [0,\infty)}$ be a solution to the harmonic heat flow equation. Assume that the maps $f_t$ converge to a limiting harmonic map $f_\infty \colon M \to N$, as $t\to\infty$, and assume that $f_\infty$ is a non-degenerate critical point of the Dirichlet energy functional. Then there exist constants $a,b>0$ such that
\[
\norm*{\der{f_t}{t}}_{L^2(f_t^*TN)} \leq a \cdot e^{-b \cdot t}
\]
for all $t\geq 0$. Moreover, the exponential decay rate (the constant $b$) depends only on $f_\infty$.
\end{theorem}

The exponential convergence rate of the harmonic heat flow has been observed before in several different settings. For example, in \cite{Topping} Topping proved that the harmonic heat flow for maps between 2-spheres converges exponentially fast in $L^2$ as $t\to\infty$. Similarly, in \cite{Wang} it is shown that the heat flow for mappings from the unit disk in $\R^2$ into closed Riemannian manifolds converges exponentially fast in $H^1$ when we assume that the Dirichlet energy along the heat flow is small. 

Our result shows that this exponential convergence behaviour is actually present in a large class of examples. For instance, if $(N,h)$ has negative curvature, then any harmonic map into $N$ that does not map into the image of a geodesic is a non-degenerate critical point of the energy. Another example is provided by equivariant harmonic maps mapping into symmetric spaces of non-compact type. A result of Sunada (\cite{Sunada}) implies that such harmonic maps are non-degenerate critical points of the energy if and only if they are unique (see \cite[Lemma 2.1]{Slegers1}).

As a corollary to \Cref{thm:exponentialconvergence} we obtain that the Dirichlet energies along the harmonic heat flow also converge exponentially fast. For a smooth map $f\colon (M,g) \to (N,h)$ we denote by $E(f)$ its Dirichlet energy (see \Cref{sec:heatflowpreliminaries}).

\begin{corollary}\label{cor:convergenceenergy}
Let $(f_t)_{t\in[0,\infty)}$,  $f_\infty$ and $b>0$ be as in $\Cref{thm:exponentialconvergence}$. Then there exists a constant $a'>0$ such that for all $t\geq 0$ we have
\[
\abs{E(f_t) - E(f_\infty)} \leq a'\cdot e^{-2b\cdot t}.
\]
\end{corollary}

\begin{acknowledgements}
The author wishes to thank Prof. Ursula Hamenst\"adt for many useful discussions. The author was supported by the IMPRS graduate program of the Max Planck Institute for Mathematics.
\end{acknowledgements}

\section{Preliminaries}\label{sec:heatflowpreliminaries}
We briefly introduce the concepts related to harmonic maps that we will need in our proof. We follow mostly the presentation given in \cite{EellsLemaireSelectedTopics} (see also \cite{EellsSampson}).

Let $(M,g)$ and $(N,h)$ be Riemannian manifolds and assume $M$ is compact. For any vector bundle $E\to M$ we denote by $\Gamma^k(E)$ the Banach space of $k$-times continuously differentiable sections of $E$. For any smooth map $f \colon M \to N$ let us denote by $\nabla$ the pullback connection on $f^*TN \to M$ induced by the Levi-Civita connection of $N$. By taking the tensor product with the Levi-Civita connection on $M$ we obtain an induced connection on the bundle $T^*M\otimes f^*TN$ which we will also denote by $\nabla$.

A smooth map $f \colon (M,g) \to (N,h)$ is a \textit{harmonic map} if it is a critical point of the Dirichlet energy
\[
E(f) = \frac{1}{2} \int_M \norm{df}^2 \vol_g.
\]
Here we consider $df$ as a section of the bundle $T^*M \otimes f^*TN$ that is equipped with the metric induced by the metrics $g$ and $h$. The \textit{tension field} of $f$ is the smooth section of $f^*TN$ that is defined as
\[
\tau(f) = \tr_g \nabla df = \sum_{i=1}^m (\nabla_{e_i} df)(e_i)
\]
where $(e_i)_{i=1}^m$ is any local orthonormal frame of $TM$ and $\nabla$ is the connection on $T^*M\otimes f^*TN$. A map $f \colon (M,g) \to (N,h)$ is harmonic if and only if its tension field vanishes identically.

The metric $g$ on $M$ and the metric on $f^*TN$ induced by the metric on $N$ give rise to the $L^2$ inner product
\[
\inner{s}{s'}_{L^2(f^*TN)} = \int_M \inner{s(m)}{s'(m)} \vol_g(m)
\]
for $s,s'\in \Gamma^0(f^*TN)$. The space $L^2(f^*TN)$ is defined to be the completion of $\Gamma^0(f^*TN)$ with respect to this inner product.

The Laplace operator induced by the pullback connection $\nabla$ on $f^*TN$ is the operator $\Delta \colon \Gamma^2(f^*TN) \to \Gamma^0(f^*TN)$ that is given by
\[
\Delta s = -\tr_g \nabla^2 s = - \sum_{i=1}^m (\nabla^2 s)(e_i,e_i)
\]
for $s\in \Gamma^2(f^*TN)$ and any (local) orthonormal frame $(e_i)_{i=1}^m$ of $TM$.  
 
\begin{definition}
We define the Jacobi operator of a smooth map $f\colon M \to N$ to be the second order differential operator that acts on sections of $f^*TN$ as
\[
\mathcal{J}_f(s) = \Delta s -  \tr_g R^N(s, df \cdot)df\cdot = - \sum_{i=1}^m \left[
(\nabla^2 s)(e_i,e_i) + R^N(s, df (e_i)) df(e_i)
\right]
\]
where $s\in \Gamma^2(f^*TN)$, $R^N$ is the curvature tensor\footnote{We define the curvature tensor as $R(X,Y)Z = \nabla_X \nabla_Y Z - \nabla_Y \nabla_X Z - \nabla_{[X,Y]}Z$ which differs from the convention chosen in \cite{EellsLemaireSelectedTopics}.} of $(N,h)$ and $(e_i)_{i=1}^m$ is any (local) orthonormal fame of $TM$.
\end{definition}

We can interpret the Jacobi operator as a densely defined operator 
\[
\mathcal{J}_f \colon L^2(f^*TN) \to L^2(f^*TN).
\]
It is a linear elliptic and self-adjoint operator. Standard spectral theory for such operators implies the following facts.

\begin{proposition}\label{eq:spectraltheoryjacobioperators}
The Hilbert space $L^2(f^*TN)$ splits orthogonally into eigenspaces of $\mathcal{J}_f$. These eigenspaces are finite dimensional and consist of smooth sections. The spectrum of $\mathcal{J}_f$ is discrete and consists of real numbers. If $(N,h)$ is non-positively curved, then the eigenvalues of $\mathcal{J}_f$ are non-negative.
\end{proposition}
\begin{proof}
See \cite[Chapter IV]{Wells} (cf. \cite[Section 4]{EellsLemaireSelectedTopics}). It is proved in \cite[Proposition 1.23]{EellsLemaireSelectedTopics} that $\Delta$ is a positive operator. If $(N,h)$ is non-positively curved, then 
\[
-\tr_g \inner{R^N(s, df \cdot)df\cdot}{s} = -\sum_{i=1}^m \inner{R^N(s, df (e_i)) df(e_i)}{s} \geq 0
\]
for any $s\in \Gamma^0(f^*TN)$ and hence it follows that the eigenvalues of $\mathcal{J}_f$ are non-negative. 
\end{proof}

When $(N,h)$ has non-positive curvature it follows that each $\mathcal{J}_f$ has a well-defined lowest eigenvalue which we will denote by $\lambda_1(\mathcal{J}_f) \geq 0$. This quantity is called \textit{the spectral gap} of the operator $\mathcal{J}_f$. Using the min-max theorem the value $\lambda_1(\mathcal{J}_f)$ can alternatively be characterised as
\begin{equation}\label{eq:minmax}
\lambda_1(\mathcal{J}_t) = \min_{\substack{s\in \Gamma^2(f^*TN) \\ s\neq 0}} \frac{\inner{\mathcal{J}_f s}{s}_{L^2(f^*TN)}}{\norm{s}^2_{L^2(f^*TN)}}.
\end{equation}

If $f$ is harmonic, then the second variation of the energy at $f$ is given by
\[
\nabla^2 E(f)(s,s') = \int_{M} \left[\inner{\nabla s}{\nabla s'} - \tr_g \inner{R^n(s, df \cdot) df \cdot}{s'}\right] \vol_g = \inner{\mathcal{J}_f s}{s'}_{L^2(f^*TN)}
\]
for any $s, s'\in \Gamma^2(f^*TN)$. We stress that this equation only holds when $f$ is harmonic. A harmonic map is a non-degenerate critical point of the energy if the bilinear form $\nabla^2 E(f)$ is non-degenerate. This happens if and only if $\ker \mathcal{J}_f = 0$. In the case that $(N,h)$ has non-positive curvature this is equivalent to $\lambda_1(\mathcal{J}_f) > 0$.

As mentioned in the introduction, the existence of a solution to the harmonic heat flow equation is due to Eells and Sampson. We record the facts relevant to our proof here in the following theorem. We denote by $C^k(M,N)$ the Banach manifold of $k$-times continuously differentiable maps from $M$ to $N$.
\begin{theorem}\label{thm:heatflowfacts}
Assume $(M,g)$ is compact and $(N,h)$ is complete and of non-positive curvature. Let $f\colon M \to N$ be a smooth map. A solution $(f_t)_{t\in [0,\infty)}$ to the harmonic heat flow equation (\Cref{eq:heatflow}) exists for all time $t\geq 0$ and the map 
\[
M\times [0,\infty) \to N \colon (m,t) \mapsto f_t(m)
\]
is smooth. Moreover, if the image of this map is contained in a compact subset of $N$, then the maps $f_t$ converge, for $t\to\infty$, to a harmonic map $f_\infty$ in any space $C^k(M,N)$.
\end{theorem}
The existence and smoothness of the solution is proved in \cite{EellsSampson} (Theorem 10.C p.154 and Proposition 6.B p.135). Note that Eells and Sampson prove these theorems under an additional assumption involving restrictions on a choice of isometric embedding $N\to \R^n$. Hartman proved in \cite[Assertion (A)]{Hartman} that this assumption is redundant. Finally, the convergence statement for $t\to\infty$ is proved in \cite[Assertion (B)]{Hartman}.

\section{Continuity of the spectral gap}

Our proof of \Cref{thm:exponentialconvergence} will rely on the fact that if $(f_t)_{t\in [0,\infty)}$ is a solution to the harmonic heat flow equation, then the associated family of Jacobi operators $\mathcal{J}_{f_t}$ is (in a loose sense) a continuous family of differential operators. The primary difficulty here is that these operators act on sections of different vector bundles. We deal with this problem in \Cref{prop:spectralgapconvergence} which will be the main tool in our proof.

Let us first introduce some notation. We will consider a family of smooth maps $(f_t)_{t\in [0,1]}$ and define $F \colon M \times [0,1] \to N$ as $F(m,t) = f_t(m)$. For each $t\in [0,1]$ we denote $E_t = f_t^* TN$ and $\mathcal{J}_t = \mathcal{J}_{f_t}$.

\begin{proposition}\label{prop:spectralgapconvergence}
Assume $F \colon M \times [0,1] \to N$ (as above) is continuous, each $f_t \colon M \to N$ is smooth and $[0,1] \to C^3(M,N) \colon t\mapsto f_t$ is continuous. Then
\[
\liminf_{t\to 0} \lambda_1(\mathcal{J}_t) \geq \lambda_1(\mathcal{J}_0).
\]
\end{proposition}
\begin{remark}
As we will see in the proof of this proposition, the statement is easily generalised to $\liminf_{t\to t_0} \lambda_1(\mathcal{J}_t) \geq \lambda_1(\mathcal{J}_{t_0})$ for $t_0 \in [0,1]$ (the choice of $t_0 = 0$ is in no way special). This means that the function $t\mapsto \lambda_1(\mathcal{J}_t)$ is lower semicontinuous. Because we don't need this full statement in our proof, we will restrict ourselves, for notational convenience, to $t_0 = 0$.
\end{remark}

As mentioned before, our main difficulty is that the differential operators $\mathcal{J}_t$ do not act on sections of the same vector bundle. To address this we first construct (local) homomorphisms between $E_t$ and $E_0$ which will allow us to locally identify these bundles.

Throughout this section we will consider the vector bundles $E_t = f_t^*TN$ as a subset of the larger vector bundle $F^*TN$ by identifying $M$ with $M\times \{t\} \subset M\times [0,1]$. Let us consider a chart $U$ of $M$ and a chart $V$ of $N$ such that for some $\epsilon>0$ the set $U\times [0,\epsilon)$ is mapped into $V$ by $F$. We will call such charts \textit{adapted charts}. To a pair of adapted charts we will associate, for $t\in [0,\epsilon)$, homomorphisms $\psi_t \colon E_t\vert_{U} \to E_0\vert_{U}$ as follows. Let us denote by $(y^\alpha)_{\alpha=1}^n$ the coordinates of the chart $V\subset N$. First, we note that $(E_\alpha)_{\alpha=1}^n$, with $E_\alpha = F^* \parder{}{y^\alpha}$, is a local frame of $F^*TN$ over $U\times [0,\epsilon)$. Furthermore, the sections $E_\alpha(\cdot, t)$ provide a frame of $E_t\vert_{U}$ for any fixed $t\in [0,\epsilon)$. If we write\footnote{Throughout this text we will use the Einstein summation convention.} an element $v\in E_t\vert_U$ as $v = v^\alpha E_\alpha(x,t)$ for some $x\in U$, then we define the map $\psi_t \colon E_t\vert_{U} \to E_0\vert_{U}$ as
\[
\psi_t(v^\alpha E_\alpha(x,t)) = v^\alpha E_\alpha(x,0).
\]
We note that for $t = 0$ we have $\psi_0 = \id$ hence, by continuity, $\psi_t$ is an isomorphism for any $t\in [0,\epsilon)$ if we take $\epsilon>0$ small enough (after possibly shrinking $U$). 

Because $M$ is compact, it can be covered by a finite set of adapted charts. More precisely, there exists an $\epsilon>0$, a finite set of charts $\{\widetilde{U}_1, \hdots, \widetilde{U}_r\}$ of $M$ and charts $\{V_1, \hdots, V_r\}$ of $N$ such that $F$ maps each $\widetilde{U}_p\times[0,\epsilon)$ into $V_p$. Let us denote by $\psi_{t,p} \colon E_t\vert_{\widetilde{U}_p} \to E_0\vert_{\widetilde{U}_p}$ the homomorphisms associated to each pair $(\widetilde{U}_p, V_p)$ of adapted charts. 

Before we proceed to the proof of \Cref{prop:spectralgapconvergence}, we will first use our choice of adapted charts to define $C^k$ norms on the spaces $\Gamma^k(E_t)$ which will be particularly well-adjusted to our arguments. Fix a $p\in \{1,\hdots, r\}$, let $(x_i)_{i=1}^m$ be the coordinates of the chart $\widetilde{U}_p \subset M$ and let $(y^\alpha)_{\alpha=1}^n$ be the coordinates of the chart $V_p \subset N$. We set, as before, $E_\alpha = F^*\parder{}{y^\alpha}$. By shrinking the open sets $\widetilde{U}_p$ slightly we can find precompact open subsets $U_p\subset \widetilde{U}_p$ such that the sets $\{U_p\}_{p=1}^r$ still cover $M$. A section $s\in \Gamma^k(E_t)$ can, locally on $\widetilde{U}_p$, be written as $s=s^\alpha E_\alpha(\cdot, t)$. Using this notation, we define, for $k\in \N$ and $t\in [0,\epsilon)$, the seminorms $\norm{\cdot}_{\Gamma^k(\overline{U}_p;E_t)}$ on $\Gamma^k(E_t)$ as
\[
\norm{s}_{\Gamma^k(\overline{U}_p;E_t)} = \sup \left\{
\abs*{\parder{^{\abs{\mu}}}{x^\mu} s^\alpha(x)} \colon x \in \overline{U}_p, 1\leq\alpha\leq n, \abs{\mu} \leq k
\right\}.
\]
Here $\mu=(\mu_1, \hdots, \mu_m)$ is a multi-index and $\parder{^{\abs{\mu}}}{x^\mu} = \parder{^{\mu_1}}{x_1^{\mu_1}}\cdots \parder{^{\mu_m}}{x_1^{\mu_m}}$. This expression is finite because $\overline{U}_p$ is compact in $\widetilde{U}_p$. We now define the norm $\norm{\cdot}_{\Gamma^k(E_t)}$ on $\Gamma^k(E_t)$ as
\[
\norm{s}_{\Gamma^k(E_t)} = \max_{p=1,\hdots, r} \norm{s}_{\Gamma^k(\overline{U}_p;E_t)}.
\]
These norms induce the usual Banach space structure on the spaces $\Gamma^k(E_t)$.

For any of the sets $U_p\subset M$, with $p=1,\hdots, r$, we will denote by $\Gamma^k(\overline{U}_p; E_t)$ the Banach space of sections of $E_t$ over $\overline{U}_p$ that extend to $k$-times differentiable sections over some open set containing $\overline{U}_p$. On this space $\norm{\cdot}_{\Gamma^k(\overline{U}_p;E_t)}$ defines a Banach norm.

By inspecting the definition of the (local) homomorphisms $\psi_{t,p} \colon E_t\vert_{\widetilde{U}_p} \to E_0\vert_{\widetilde{U}_p}$ and the seminorms $\norm{\cdot}_{\Gamma^k(\overline{U}_p;E_t)}$ we observe the following. For all $k\in \N$ and $t\in [0,\epsilon)$, if $s\in \Gamma^k(\overline{U}_p;E_t)$ is a section, then
\begin{equation}\label{eq:welladjustedness}
\norm{\psi_{t,p}(s)}_{\Gamma^k(\overline{U}_p;E_0)} = \norm{s}_{\Gamma^k(\overline{U}_p;E_t)}.
\end{equation}
We will use this compatibility between the homomorphisms and seminorms in our proof of \Cref{prop:spectralgapconvergence}.
\begin{proof}[Proof of \Cref{prop:spectralgapconvergence}]
Let the adjusted charts $(\widetilde{U}_p, V_p)$, associated homomorphisms $\psi_{t,p} \colon E_t\vert_{\widetilde{U}_p} \to E_0\vert_{\widetilde{U}_p}$ and choice of precompact opens $U_p\subset \widetilde{U}_p$ be as above.

Let us denote $\lambda = \liminf_{t\to 0} \lambda_1(\mathcal{J}_{t})$. There exists a sequence $(t_u)_{u\in \N} \subset [0,\epsilon)$ such that $t_u \to 0$ as $u\to\infty$ and
\[
\lim_{u\to\infty} \lambda_1(\mathcal{J}_{t_u}) = \lambda =  \liminf_{t\to 0} \lambda_1(\mathcal{J}_{t}).
\]
It follows from \Cref{eq:spectraltheoryjacobioperators} that for each $u\in \N$ there exists a smooth eigensection $s_u \in \Gamma^\infty(E_t)$ such that $\mathcal{J}_{t_u} s_u = \lambda_1(\mathcal{J}_{t_u}) \cdot s_u$. We normalise such that $\norm{s_u}_{\Gamma^0(E_t)} = 1$ for all $u\in \N$.

For $p = 1,\hdots, r$ we denote $\sigma_{u,p} = \psi_{t_u,p}(s_u\vert_{\widetilde{U}_p}) \in \Gamma^\infty(\widetilde{U}_p;E_0)$. Our proof will rely on the following two lemmas.

\begin{lemma}\label{lem:convergence}
There exists a subsequence $(u_k)_{k\in \N}\subset \N$ such that for each $p = 1,\hdots, r$ the sequence $(\sigma_{u_k,p})_{k\in\N}$ converges in $\Gamma^2(\overline{U}_p;E_0)$ to a limiting section $\sigma_p \in \Gamma^2(\overline{U}_p;E_0)$. At least one of these limiting sections is not the zero section. Moreover, for all $p,q = 1,\hdots, r$ the sections $\sigma_p$ and $\sigma_q$ coincide on $\overline{U}_p \cap \overline{U}_q$.
\end{lemma}

In the second lemma we consider the operator $\mathcal{J}_0$ restricted to the open sets $U_p$. Since the Jacobi operators $\mathcal{J}_t$ are ordinary differential operators, it follows that the value of $\mathcal{J}_t s$ at a point in $M$ depends only on the germ of the section $s$ at that point. Hence, we can apply $\mathcal{J}_t$ also to sections that are not globally defined.

\begin{lemma}\label{lem:limitiseigenvector}
Consider the limiting sections $\sigma_p \in \Gamma^2(\overline{U}_p;E_0)$ as defined in \Cref{lem:convergence}. For all $p=1,\hdots, r$ we have on $U_p$ that
\[
\mathcal{J}_{0} \sigma_p = \lambda \cdot \sigma_p.
\]
\end{lemma}
We postpone the proof of these two lemmas and first finish proof of \Cref{prop:spectralgapconvergence}.

It follows from the last statement of \Cref{lem:convergence} that we can patch the limiting sections $\sigma_p$ together to obtain a well-defined global limiting section $\sigma\in \Gamma^2(E_0)$. More precisely, we let $\sigma \in \Gamma^2(E_0)$ be the section that on each $\overline{U}_p\subset M$ is given by $\sigma\vert_{\overline{U}_p} = \sigma_p$. Note that the sets $\overline{U}_p$ cover $M$ and that by \Cref{lem:convergence} the section is well-defined on intersections $\overline{U}_p \cap \overline{U}_q$. Because at least one of the limiting sections $\sigma_p$ does not vanish, it follows that $\sigma$ is not the zero section.

Now \Cref{lem:limitiseigenvector} implies that $\sigma$ is an eigensection of $\mathcal{J}_0$. Namely, we have 
\[
\mathcal{J}_0 \sigma = \lambda \cdot \sigma
\]
because this holds on each subset $U_p \subset M$. It follows that $\lambda$ is an eigenvalue of $\mathcal{J}_0$ and hence that
\[
\lambda_1(\mathcal{J}_0) \leq \lambda = \liminf_{t\to 0} \lambda_1(\mathcal{J}_{t}).
\]
\end{proof}

We now prove \Cref{lem:convergence} and \Cref{lem:limitiseigenvector}. The proofs of these lemmas will rely on the fact that, in suitably chosen local coordinates, the coefficients of the differential operators $\mathcal{J}_t$ depend continuously on $t$.

Let us first introduce the necessary notation. Let $(\widetilde{U}_p, V_p)$ be a pair of adapted charts as before, $(x^i)_{i=1}^m$ the coordinates on $\widetilde{U}_p$ and $(y^\alpha)_{\alpha=1}^n$ the coordinates on $V_p$. We put again $E_\alpha = F^*\parder{}{y^\alpha}$. The Jacobi operators $\mathcal{J}_t$ are second order differential operators. Hence, in local coordinates they can be written as
\begin{equation}\label{eq:generaldiffop}
\mathcal{J}_{t} s(x) = \left\lbrace A_{\alpha}^{ij,\gamma}(x,t) \parder{^2 s^\alpha}{x^i x^j}(x) + B_{\alpha}^{i, \gamma}(x,t) \parder{s^\alpha}{x^i} (x) + C^\gamma_\alpha(x,t) s^\alpha(x) \right\rbrace E_\gamma(x,t),
\end{equation}
where $A_{\alpha}^{ij,\gamma}, B_{\alpha}^{i, \gamma}, C^\gamma_\alpha \colon \widetilde{U}_p \times [0,\epsilon) \to \R$ are suitable coefficient functions. Here we write any section $s$ of $E_t$ over $\widetilde{U}_p$ as $s=s^\alpha E_\alpha(\cdot, t)$.

Our proofs of \Cref{lem:convergence} and \Cref{lem:limitiseigenvector} are based on the following observation.

\begin{lemma}\label{lem:continuousdependencecoeffs}
Let $U'\subset \widetilde{U}_p$ be a precompact open subset. For all $i,j=1,\hdots,m$ and $\alpha, \gamma = 1,\hdots, n$, we have that the maps $t\mapsto A_\alpha^{ij,\gamma}(\cdot, t), t\mapsto B_\alpha^{i,\gamma}(\cdot, t)$ and $t\mapsto C_\alpha^\gamma(\cdot, t)$ are continuous mappings from $[0,1]$ into $C^1(\overline{U'})$.
\end{lemma}

\begin{proof}
Denote by $g^{ij}$ the coefficients of the inverse of the metric tensor $g$ with respect to the coordinates $(x^i)_{i=1}^m$ and by $\presuperscript{M}{\Gamma}^{k}_{ij}$ the Christoffel symbols of the Levi-Civita connection of $(M,g)$. The Jacobi operators are expressed locally as
\begin{align*}
\mathcal{J}_t s &= \Delta s - \tr_g R^N(s, df\cdot) df\cdot\\
&= -g^{ij}\left\{ \nabla_{\parder{}{x^i}} \nabla_{\parder{}{x^j}} s - \presuperscript{M}{\Gamma}^{k}_{ij} \nabla_{\parder{}{x^k}} s + R^N\left(s, \parder{f}{x^i}\right)\parder{f}{x^j}\right\},
\end{align*}
with $s \in \Gamma^2(\overline{U}_p; E_t)$. Recall that $\nabla$ is the pullback connection on the bundle $E_t = f_t^*TN$. Let us denote by  $\presuperscript{N}{\Gamma}^{\gamma}_{\alpha \beta}$ the Christoffel symbols of the Levi-Civita connection of $(N,h)$ on the chart $V_p$. Then, for any $s = s^\alpha E_\alpha(\cdot, t) \in \Gamma^1(\widetilde{U}_p; E_t)$, we can write the pullback connection as
\[
\nabla_{\parder{}{x^i}} s(x) = \parder{s^\alpha}{x^i}(x) E_\alpha(x, t) + s^\alpha(x) \parder{f_t^\beta}{x^i}(x)\cdot \presuperscript{N}{\Gamma}^{\gamma}_{\alpha \beta}(f_t(x)) \cdot E_\gamma(x,t).
\]
The coefficient functions $A_\alpha^{ij,\gamma}, B_\alpha^{i,\gamma}, C_\alpha^\gamma$ can be calculated by filling in this expression for the connection $\nabla$ into the local expression for the Jacobi operators. It follows that these functions can be expressed entirely in terms of the quantities 
\[
g^{ij}, \parder{f_t^\beta}{x^i}, \presuperscript{M}{\Gamma}^{k}_{ij}, (R^N)_{\alpha\beta\gamma}^\delta \circ f_t \text{ and }\presuperscript{N}{\Gamma}^{\gamma}_{\alpha \beta}\circ f_t
\]
and their first derivatives. Here $(R^N)_{\alpha\beta\gamma}^\delta$ denote the coefficients of the Riemann curvature tensor $R^N$ in the coordinates on $V_p$. As a result, in the expression for the coefficient functions only spatial derivatives of the functions $f_t$ up to second order appear. The statement of the lemma now follows immediately from our assumption that $[0,1]\to C^3(M,N) \colon t\mapsto f_t$ is a continuous mapping.
\end{proof}

We can now prove \Cref{lem:convergence}.

\begin{proof}[Proof of \Cref{lem:convergence}]
Fix a $p\in \{1,\hdots, r\}$. Let us write $s_u = s_u^\alpha E_\alpha(\cdot, t)$ on $\widetilde{U}_p$. Because each $s_u$ is an eigensection of the Jacobi operator $\mathcal{J}_{t_u}$, we find that they satisfy
\begin{equation}\label{eq:pdesystem}
\left[ \mathcal{J}_{t_u} - \lambda_1(\mathcal{J}_{t_u})\right] s_u = 0.
\end{equation}
Hence, on $\widetilde{U}_p$ the coefficients $(s_u^\alpha)_{\alpha=1}^n$ satisfy a second order linear elliptic system of differential equations. We will use Schauder estimates to obtain a uniform bound on the $C^{2,\mu}$-H\"older norm of these coefficients. To this end we will apply the results of \cite{Morrey}.

The system of differential equations in \Cref{eq:pdesystem} is elliptic because the Jacobi operators are elliptic differential operators. The bounds on the H\"older norms of solutions to this equation that are provided by Morrey's results depend on a uniform ellipticity constant which in Morrey's paper is denoted $M$ (defined in \cite[Equation 1.6]{Morrey}). This constant depends only on the coefficients of the second order part of the system in \Cref{eq:pdesystem}. That is, it depends only on the coefficients $A_{\alpha}^{ij,\gamma}$. Because, by \Cref{lem:continuousdependencecoeffs}, these coefficient functions depend continuously on $t$, it follows that the constant $M$ can be taken uniformly over $u\in \N$.

Take a precompact open $U'\subset \widetilde{U}_p$ such that $\overline{U}_p \subset U'\subset \overline{U'} \subset \widetilde{U}_p$. The coefficients of the system of differential equations in \Cref{eq:pdesystem} are a combination of the coefficients of $\mathcal{J}_{t_u}$ and the constant term $\lambda_1(\mathcal{J}_{t_u})$. It follows from \Cref{lem:continuousdependencecoeffs} that the $C^{0,\mu}$-H\"older norms (even $C^1$ norms) of the coefficients of $\mathcal{J}_{t_u}$ can be bounded uniformly in $u$. The constant term $\lambda_1(\mathcal{J}_{t_u})$ can also be bounded uniformly in $u$, since the sequence $(\lambda_1(\mathcal{J}_{t_u}))_{u\in \N}$ is convergent. So the coefficients of the system of differential equations in \Cref{eq:pdesystem} have uniformly (in $u$) bounded $C^{0,\mu}$-H\"older norms. Moreover, because we normalised the sections $s_u$ such that $\norm{s_u}_{\Gamma^0(E_t)} = 1$, it follows that the $C^0$ norm (and hence also the $L^2$ norm) of the coefficients $s^\alpha_u$ is also bounded uniformly in $u$. We now apply \cite[Theorem 4.7]{Morrey} (with $G=U', G_1 = U_p$, in the notation of that paper) to conclude that on $\overline{U}_p$ the $C^{2,\mu}$-H\"older norms of the coefficients $s_u^\alpha$ are uniformly bounded in $u$. 

We recall the notation $\sigma_{u,p} = \psi_{t_u,p}(s_u\vert_{\widetilde{U}_p})$. It follows from the definition of the homomorphisms $\psi_{t,p}$ that $s_u$ and $\sigma_{u,p}$ have the same coefficients on $\widetilde{U}_p$. Namely, if we write $\sigma_{u,p} = \sigma_{u,p}^\alpha E_\alpha(\cdot, 0)$, then $s_{u}^\alpha = \sigma_{u,p}^\alpha$ for $\alpha = 1,\hdots, n$. Hence, also the $C^{2,\mu}$-H\"older norms of the coefficients $\sigma_{u,p}^\alpha$ are uniformly bounded. It now follows from the Arzel\`a-Ascoli theorem that there exists a subsequence of $(\sigma_{u,p})_{u\in \N}$ that converges in $\Gamma^2(\overline{U}_p; E_0)$ to a limiting section. We denote this limiting section by $\sigma_p$. By choosing subsequent refinements of the subsequence we can arrange for this to hold for each $p=1,\hdots, r$. We denote the indices of this subsequence by $(u_k)_{k\in \N} \subset \N$.

We now prove that is it not possible that all limiting sections $\sigma_p$ vanish identically. If this was the case, and all sections $\sigma_p$ vanish, then this would imply $\norm{\sigma_{u_k,p}}_{\Gamma^0(\overline{U}_p;E_0)} \to 0$ as $k\to\infty$ for all $p=1,\hdots, r$. However, this contracts that for all $u\in \N$ we have, by \Cref{eq:welladjustedness}, that
\[
\max_{p=1,\hdots, r} \norm{\sigma_{u,p}}_{\Gamma^0(\overline{U}_p;E_0)} = \max_{p=1,\hdots, r} \norm{s_u}_{\Gamma^0(\overline{U}_p;E_t)} = \norm{s_u}_{\Gamma^0(E_t)} = 1.
\]

Finally, we prove the last statement of the lemma. Let $(\widetilde{U}_p, V_p)$ and $(\widetilde{U}_q, V_q)$ be two pairs of adapted charts with corresponding local homomorphisms $\psi_{t,p}$ and $\psi_{t,q}$. Recall that the maps  $\psi_{t,p} \colon E_t\vert_{\widetilde{U}_p} \to E_0\vert_{\widetilde{U}_p}$ are isomorphisms for $t$ small enough. It can be easily seen from the definition of these homomorphisms that, on the compact set $\overline{U}_p\cap \overline{U}_q$, the maps
\[
\psi_{t,q}\circ \psi_{t,p}^{-1} \colon E_0\vert_{\overline{U}_p\cap \overline{U}_q} \to E_0\vert_{\overline{U}_p\cap \overline{U}_q}
\]
converge uniformly to the identity map as $t\to 0$. It follows that
\begin{align*}
\sigma_p\vert_{\overline{U}_p \cap \overline{U}_q} &= \lim_{k\to\infty} \psi_{t_{u_k},p}(s_{u_k}\vert_{\overline{U}_p \cap \overline{U}_q})\\
&=\lim_{k\to\infty} \psi_{t_{u_k},q}\circ \psi_{t_{u_k},p}^{-1} \circ \psi_{t_{u_k},p}(s_{u_k}\vert_{\overline{U}_p \cap \overline{U}_q})\\ &= \lim_{k\to\infty} \psi_{t_{u_k},q}(s_{u_k}\vert_{\overline{U}_p \cap \overline{U}_q})\\
&= \sigma_q\vert_{\overline{U}_p \cap \overline{U}_q},
\end{align*}
where the limits are taken in $\Gamma^0(\overline{U}_p\cap \overline{U}_q;E_0)$.
\end{proof}

We finish this section with the proof of \Cref{lem:limitiseigenvector}.
\begin{proof}[Proof of \Cref{lem:limitiseigenvector}]
Fix a $p\in \{1,\hdots, p\}$. Let $(\widetilde{U}_p,V_p)$ be a pair of adapted charts and let the homomorphisms $\psi_{t,p}$ and the frame $(E_\alpha)_{\alpha=1}^n$ be as before.

We claim that 
\begin{equation}\label{eq:operatornormconverges}
\norm{\psi_{t,p}\circ \mathcal{J}_{t} - \mathcal{J}_0 \circ \psi_{t,p}}_{\text{op}} \to 0 \text{ as } t\to 0.
\end{equation}
Here, $\norm{\cdot}_{\text{op}}$ is the operator norm on the space of bounded linear operators from $\Gamma^2(\overline{U}_p; E_t)$ to $\Gamma^0(\overline{U}_p; E_0)$ (equipped with the norms $\norm{\cdot}_{\Gamma^2(\overline{U}_p; E_t)}$ and $\norm{\cdot}_{\Gamma^0(\overline{U}_p; E_0)}$ respectively).

We denote
\begin{align*}
a_{\alpha}^{ij,\gamma}(x,t) &= A^{ij,\gamma}_{\alpha}(x,t) - A^{ij,\gamma}_{\alpha}(x,0)\\
b_{\alpha}^{i,\gamma}(x,t) &= B^{i,\gamma}_{\alpha}(x,t) - B^{i,\gamma}_{\alpha}(x,0)\\
c_{\alpha}^{\gamma}(x,t) &= C^{\gamma}_{\alpha}(x,t) - C^{\gamma}_{\alpha}(x,0).
\end{align*}
Then, for a section $s = s^\alpha E_\alpha(\cdot, t) \in \Gamma^2(\overline{U}_p; E_t)$, we have
\begin{align*}
[ \psi_{t,p}\circ \mathcal{J}_{t} &- \mathcal{J}_0 \circ \psi_{t,p}]s(x) \\&= \left\{a_{\alpha}^{ij,\gamma}(x,t) \parder{^2 s^\alpha}{x^i x^j}(x) + b_\alpha^{i,\gamma}(x,t) \parder{s^\alpha}{x^i}(x) + c_\alpha^\gamma(x,t) s^\alpha(x)\right\} E_\alpha(x,0).
\end{align*}
From this expression follows that
\[
\norm{\psi_{t,p}\circ \mathcal{J}_{t} - \mathcal{J}_0 \circ \psi_{t,p}}_{\text{op}} \leq \sum_{i,j,\alpha,\gamma} \norm{a_{\alpha}^{ij,\gamma}}_{C^0(\overline{U}_p)} + \sum_{i,\alpha,\gamma} \norm{b_{\alpha}^{i,\gamma}}_{C^0(\overline{U}_p)} + \sum_{\alpha,\gamma} \norm{c_{\alpha}^{\gamma}}_{C^0(\overline{U}_p)}.
\]
Our claim now immediately implied by the results of \Cref{lem:continuousdependencecoeffs}.

We use the notation $(u_k)_{k\in \N}$ and $\sigma_{u,p}$ as in \Cref{lem:convergence}. From that lemma follows that $\sigma_{u_k, p} \to \sigma_p$ in $\Gamma^2(\overline{U}_p;E_0)$. We use this to find
\[
\mathcal{J}_0 \sigma_p = \lim_{k\to\infty} \mathcal{J}_0 \sigma_{u_k,p} = \lim_{k\to\infty} \mathcal{J}_0 \psi_{t_{u_k},p}(s_{u_k}\vert_{\overline{U}_p}).
\]
From \Cref{eq:operatornormconverges} follows that
\[
\mathcal{J}_0 \sigma_p = \lim_{k\to\infty} \mathcal{J}_0 \psi_{t_{u_k},p}(s_{u_k}\vert_{\overline{U}_p}) = \lim_{k\to\infty} \psi_{t_{u_k},p} (\mathcal{J}_{t_{u_k}} s_{u_k}\vert_{\overline{U}_p}).
\]
Here we used that $\norm{s_{u_k}}_{\Gamma^2(\overline{U}_p; E_t)} = \norm{\sigma_{u_k}}_{\Gamma^2(\overline{U}_p; E_0)}$ remains bounded uniformly in $k$. Finally, using the fact that the sections $s_u$ are eigensections of the operators $\mathcal{J}_{t_u}$ gives
\[
\mathcal{J}_0 \sigma_p = \lim_{k\to\infty} \psi_{t_{u_k},p} (\mathcal{J}_{t_{u_k}} s_{u_k}\vert_{\overline{U}_p}) = \lim_{k\to\infty} \lambda(\mathcal{J}_{t_{u_k}}) \cdot \psi_{t_{u_k},p}(s_{u_k}\vert_{\overline{U}_p}) = \lambda \cdot \sigma_p
\]
because, by definition, $\lambda = \lim_{u\to\infty} \lambda_1(\mathcal{J}_{t_u})$.
\end{proof}

\section{Proof of \Cref{thm:exponentialconvergence}}
Our proof of \Cref{thm:exponentialconvergence} will rely on the fact that the Jacobi operator of the maps $f_t$ appears in the evolution equation for the quantity $\tau(f_t)$. Recall the notation $E_t = f_t^*TN$.

\begin{lemma}\label{lem:evolutionequation}
Assume the family of maps $(f_t)_{t\in [0,\infty)}$ satisfies the harmonic heat flow equation. Then
\[
\frac{1}{2} \der{}{t} \norm{\tau(f_t)}^2_{L^2(E_t)} = -\inner{\mathcal{J}_{f_t} \tau(f_t)}{\tau(f_t)}_{L^2(E_t)}.
\]
\end{lemma}

\begin{proof}
Assume $(x^i)_{i=1}^m$ are Riemannian normal coordinates around a point $x \in M$. In the following calculation we will consider the expression $\parder{f_t}{x^\alpha}$ as a local section of $f_t^*TN$. Because we are working in normal coordinates around $x$, we have that
\[
\tau(f_t) \vert_x = \tr_{g} \nabla df\vert_x = \nabla_{\parder{}{x^i}}(\parder{f}{x^i})\Big\vert_x.
\]
We use this to find that at the point $x$ and for any $t\geq 0$ we have
\begin{align*}
\nabla_{\parder{}{t}}\tau(f_t) &= \nabla_{\parder{}{t}}\left(\nabla_{\parder{}{x^i}}\left(\parder{f}{x^i}\right)\right)\\&= R^N \left(\parder{f}{t}, \parder{f}{x^i} \right)\parder{f}{x^i} + \nabla_{\parder{}{x^i}}\nabla_{\parder{}{x^i}}  \left(\parder{f}{t}\right)\\
&= -\Delta \tau(f_t) + \tr_{g} R^N(\tau(f_t) , df \cdot) df \cdot = - \mathcal{J}_{f_t} \tau(f_t).
\end{align*}
To get the second equality we used that $\nabla_{\parder{}{t}}\parder{f}{x^i}=\nabla_{\parder{}{x^i}}\parder{f}{t}$ (see \cite[p.5]{EellsLemaireSelectedTopics}). Because $x\in M$ was arbitrary, we conclude that his equality holds everywhere. We use this to find that
\[
\frac{1}{2} \der{}{t} \norm{\tau(f_t)}^2_{L^2(E_t)} = \inner{\nabla_{\parder{}{t}}\tau(f_t)}{\tau(f_t)}_{L^2(E_t)} = - \inner{\mathcal{J}_{f_t} \tau(f_t)}{\tau(f_t)}_{L^2(E_t)}.
\]
\end{proof}

We can now give a proof of \Cref{thm:exponentialconvergence}.

\begin{proof}[Proof of \Cref{thm:exponentialconvergence}]
We apply \Cref{prop:spectralgapconvergence} to the family of maps $(f_t)_{t\in [0,\infty]}$. For this we pick some homeomorphism between $[0,\infty]$ and $[0,1]$ (mapping $\infty$ to $0$) so we can view the heat flow as a family of maps $(f_t)_{t\in [0,1]}$ indexed by $t\in [0,1]$. It then follows from \Cref{thm:heatflowfacts} that this family of maps satisfies the assumptions of \Cref{prop:spectralgapconvergence}. From this proposition follows that
\[
\liminf_{t\to\infty} \lambda_1(\mathcal{J}_{f_t}) \geq \lambda_1(\mathcal{J}_{f_\infty}).
\]
By assumption $f_\infty$ is a non-degenerate critical point of the energy so $\lambda_1(\mathcal{J}_{f_\infty}) > 0$. Put $b = \lambda_1(\mathcal{J}_{f_\infty})/2 > 0$. Then, for $t \geq t_0$ large enough we have $\lambda_1(\mathcal{J}_{f_t}) \geq b$. Using \Cref{lem:evolutionequation} and \Cref{eq:minmax} we see that for such $t\geq t_0$,
\[
\der{}{t} \norm{\tau(f_t)}_{L^2(E_t)}^2 = -2 \inner{\mathcal{J}_{f_t} \tau(f_t)}{\tau(f_t)}_{L^2(E_t)} \leq -2b \cdot \norm{\tau(f_t)}^2_{L^2(E_t)}.
\]
Gr\"onwalls's inequality (\cite{Gronwall}) yields that
\[
\norm{\tau(f_t)}_{L^2(E_t)}^2 \leq \norm{\tau(f_{t_0})}_{L^2(E_t)}^2 \cdot e^{-2b\cdot t}
\]
for $t\geq t_0$. So if we pick $a>0$ large enough, then
\[
\norm*{\der{f_t}{t}}_{L^2(E_t)} = \norm{\tau(f_t)}_{L^2(E_t)} \leq a\cdot e^{-b\cdot t}
\]
for all $t\geq 0$.
\end{proof}

We end with the proof of \Cref{cor:convergenceenergy}.
\begin{proof}[Proof of \Cref{cor:convergenceenergy}]
The evolution of the energy $E(f_t)$ along the harmonic heat flow is governed by the equation
\[
\der{}{t} E(f_t) = - \int_{M} \norm*{\tau(f_t)}^2 \vol_g = - \norm*{\der{f_t}{t}}^2_{L^2(E_t)}
\]
(see \cite[\S 6.C]{EellsSampson}). Applying the estimate of \Cref{thm:exponentialconvergence} gives
\[
\abs{E(f_t) - E(f_\infty)} = \int_{t}^\infty \norm*{\der{f_t}{t}}^2_{L^2(E_t)} dt \leq a\cdot \int_t^\infty e^{-2b\cdot t} dt \leq a'\cdot e^{-2b \cdot t}
\]
with $a'= a/(2b)$.
\end{proof}

\bibliographystyle{alpha}
\bibliography{bibliography}

\end{document}